\documentclass[11pt]{amsart}

\usepackage{amsmath,amssymb,amsthm,mathtools}
\usepackage[colorlinks=true,linkcolor=blue,citecolor=blue,urlcolor=blue]{hyperref}

\newtheorem{theorem}{Theorem}[section]
\newtheorem{proposition}[theorem]{Proposition}
\newtheorem{lemma}[theorem]{Lemma}
\newtheorem{corollary}[theorem]{Corollary}
\newtheorem*{maintheorem}{Main Theorem}

\theoremstyle{definition}
\newtheorem{definition}[theorem]{Definition}
\newtheorem{problem}[theorem]{Problem}
\theoremstyle{remark}
\newtheorem{remark}[theorem]{Remark}

\newcommand{\C}{\mathbb{C}}

\newcommand{\dec}{\mathrm{dec}}
\newcommand{\cb}{\mathrm{cb}}
\newcommand{\Wst}{W^*}
\newcommand{\U}{\mathcal{U}}
\newcommand{\linf}{\ell^\infty_3}
\newcommand{\iu}{\mathrm{i}}

\begin{document}

\title[Haagerup's Problem]{A solution to Haagerup's Problem on decomposable maps}

\author{Rui OKAYASU}
\address{Department of Mathematics Education, Osaka Kyoiku University, Kashiwara, Osaka 582-8582, JAPAN}
\email{rui@cc.osaka-kyoiku.ac.jp}

\subjclass[2020]{Primary 46L07; Secondary 46L10, 47L25}
\keywords{decomposable maps, completely bounded maps, injective von Neumann algebras, generator problem}
\thanks{The author is partially supported by JSPS KAKENHI Grant Number JP26K06843.}

\begin{abstract}
We prove that a von Neumann algebra $N$ is injective if and only if $\|T\|_{\mathrm{dec}}=\|T\|_{\mathrm{cb}}$ for every
linear map $T:\ell^\infty_3\to N$, answering a problem of Haagerup (1985).
\end{abstract}

\maketitle


\section{Introduction}

For $C^*$-algebras $A,B$, Haagerup \cite{Haa85} introduced the decomposable norm $\|\cdot\|_{\dec}$ on the
space of linear combinations of completely positive (CP) maps $A\to B$, and showed that
$\|T\|_{\cb}\le\|T\|_{\dec}$. When $B$ is injective, every completely bounded (CB) map $A\to B$ is decomposable with
$\|T\|_{\dec}=\|T\|_{\cb}$; this follows from the decomposition theorems of Wittstock \cite{Wittstock} and Paulsen \cite{Paulsen}, as
observed in \cite{Haa85}. His main theorem
\cite[Theorem~2.1]{Haa85} gives the converse for von Neumann algebras $N$: if
$\|T\|_{\dec}=\|T\|_{\cb}$ for every $n$ and every linear map $T:\ell^\infty_n\to N$, then $N$ is injective.
For $n=2$ the two norms always coincide \cite[Proposition~3.4]{Haa85}, and Haagerup asked:

\begin{problem}[{\cite[Problem~3.8]{Haa85}}]
Let $N$ be a von Neumann algebra such that $\|T\|_{\cb}=\|T\|_{\dec}$ for every linear map
$T:\ell^\infty_3\to N$. Is $N$ injective?
\end{problem}

\begin{definition}
A $C^*$-algebra $B$ satisfies \textbf{(H3)} if $\|T\|_{\dec}=\|T\|_{\cb}$ for every linear map
$T:\linf\to B$. Here $\|T\|_{\dec}$ is computed with respect to $B$ as the range.
\end{definition}

Pisier \cite[Remark~23.4]{Pis20} records that the problem is open and that two cases are known:
\begin{itemize}
\item the case where $N$ has a cyclic vector and $N'$ is generated by two unitaries;
\item the case where (H3) holds for all matrix algebras $M_k(N)$ (Junge--Le Merdy \cite{JLM}).
\end{itemize}
The first case ties the problem to the generator problem, and the second requires information at
all matrix levels, which (H3) does not provide. We prove:

\begin{maintheorem}\label{thm:main}
Every von Neumann algebra satisfying \textup{(H3)} is injective.
\end{maintheorem}

The proof does not use the recent solution of the generator problem \cite{OAIgen}; with it, the proof
becomes much shorter (Remark~\ref{rem:OAI}). 

\section{Preliminaries}\label{sec:known}

Let $A, B$ be $C^*$-algebras.
We denote by $\mathrm{CP}(A,B)$
the set of all CP maps from $A$ to $B$.
For a linear map $T\colon A\to B$, define
\begin{equation}\label{eq:dec}
\|T\|_{\dec}\coloneqq\inf\big\{\max\{\|S_1\|,\|S_2\|\}\big\},
\end{equation}
where the infimum is taken over all $S_1,S_2\in \mathrm{CP}(A,B)$ such that
\[
x\mapsto\begin{bmatrix} S_1(x) & T(x)\\ T(x^*)^* & S_2(x)\end{bmatrix}\in \mathrm{CP}(A,M_2(B)),
\]
with the convention $\|T\|_{\dec}=\infty$ if no such $S_1,S_2$ exist.

For $C^*$-algebras $A,B,C$, we have the following \cite[Proposition~1.3]{Haa85}:
\begin{enumerate}
\item[(D1)] $\|T\|_{\cb}\le\|T\|_{\dec}$ for every decomposable map $T\colon A\to B$;
\item[(D2)] $\|ST\|_{\dec}\le\|S\|_{\dec}\|T\|_{\dec}$ for decomposable maps $T\colon A\to B$ and $S\colon B\to C$;
\item[(D3)] $\|E\|_{\dec}=\|E\|$ for every CP map $E\colon A\to B$.
\end{enumerate}

For a von Neumann algebra $M$, 
we denote by $\U(M)$ its unitary group and
$\Wst(S)$ the von Neumann algebra generated by a subset $S$ and the unit. 
A von Neumann algebra $M$ is \emph{singly generated} if $M=\Wst(z)$ for some $z\in M$, equivalently if
$M=\Wst(h_1,h_2)$ for two self-adjoint elements $h_1,h_2$ (take $z=h_1+\iu h_2$).

A von Neumann algebra $M\subseteq B(H)$ is \emph{injective} if there is a projection of norm one from $B(H)$ onto $M$.
Recall that injectivity does not depend on the choice of a faithful normal representation of $M$.

 The next lemma is a special case of \cite[Proposition~5.7]{EL}.

\begin{lemma}\label{lem:locality}
Let $N\subseteq B(H)$ be a von Neumann algebra such that $\Wst(F)$ is injective for every finite subset
$F\subseteq\U(N)$. Then $N$ is injective.
\end{lemma}

The proposition below is an immediate consequence of the Junge--Le~Merdy theorem {\cite[Proposition 3.5]{JLM}}.
This is also how the result is quoted in \cite[Remark~10.7]{Pis19}, \cite[23.7]{Pis20}.

\begin{proposition}[Junge--Le Merdy]\label{prop:pi}
Let $M$ be a von Neumann algebra such that $M_k(M)$ satisfies \textup{(H3)} for every $k\ge1$.
Then $M$ is injective.
\end{proposition}

The following proposition is classical; we list its sources after the statement.
It is also mentioned as a known fact in \cite{Hayes}.
Moreover, it is a special case of the recently announced solution of the generator problem
\cite{OAIgen}, according to which every von Neumann algebra with separable predual is
singly generated.

\begin{proposition}\label{prop:K1}
Every injective von Neumann algebra with separable predual is singly generated.
\end{proposition}

\noindent\emph{Sources for Proposition~\ref{prop:K1}.}
Let $C$ be injective with separable predual, realized on a separable Hilbert space, and let
\[
C=\int^\oplus C(\lambda)\,d\mu(\lambda)
\] 
be its central decomposition.
\begin{enumerate}
\item By Willig \cite[Theorem~1]{Willig}, $C$ is singly generated if almost every $C(\lambda)$ is.
\item Almost every $C(\lambda)$ is injective by \cite[Proposition~6.5]{Connes}.
\item Injective factors are singly generated:
\begin{itemize}
\item factors of type I by Pearcy \cite{Pearcy};
\item properly infinite factors by Wogen \cite{Wogen};
\item injective factors of type II$_1$ are isomorphic to $R$ \cite[Theorem~7.1]{Connes}, and
$R=\overline\bigotimes(M_2,\mathrm{tr})$ is generated by a commuting sequence of type I$_2$ subfactors,
hence is singly generated by Suzuki and Sait\^o \cite[Lemma~4 and Theorem~1]{SuzukiSaito}.
\end{itemize}
\end{enumerate}

Pisier's theorem is one of the main ingredients; see {\cite[Remark~10.5]{Pis19}}, {\cite[Remark~23.4]{Pis20}}. 

\begin{theorem}[Pisier]\label{thm:pisier}
Let $M\subseteq B(H)$ be a von Neumann algebra with a cyclic vector $\Omega$, and let $w_1,w_2\in M'$ be
unitaries with $\Wst(w_1,w_2)=M'$. If $M$ satisfies \textup{(H3)}, then there is a unital CP projection
of $B(H)$ onto $M$. In particular $M$ is injective.
\end{theorem}

\begin{corollary}\label{cor:single}
Let $B$ be a singly generated von Neumann algebra with a faithful normal state $\varphi$. If $B$
satisfies \textup{(H3)}, then $B$ is injective.
\end{corollary}

\begin{proof}
Represent $B$ on $H=L^2(B,\varphi)$. Then the implementing vector $\xi_\varphi$ is cyclic and separating. Denote by $J_\varphi$ the modular conjugation. 
Since $B'=J_\varphi BJ_\varphi$,
$B=\Wst(z)$ implies $B'=\Wst(J_\varphi zJ_\varphi)$.

Write $J_\varphi zJ_\varphi=h_1+\iu h_2$ with $h_1,h_2$ self-adjoint, choose $\varepsilon>0$ with
$\varepsilon\|h_1\|,\varepsilon\|h_2\|<\pi$, and put $w_r=e^{\iu\varepsilon h_r}$ for $r=1,2$.
The spectrum of $w_r$ lies in $\{e^{\iu\theta}:|\theta|<\pi\}$, on which the principal logarithm is
continuous. Hence $h_r=(\iu\varepsilon)^{-1}\operatorname{Log}(w_r)\in\Wst(w_r)$. Therefore
\[
\Wst(w_1,w_2)=\Wst(h_1,h_2)=\Wst(J_\varphi zJ_\varphi)=B'.
\]
By Theorem~\ref{thm:pisier}, $B$ is
injective.
\end{proof}

\section{Key lemmas}\label{sec:new}

\begin{lemma}\label{lem:heredity}
Let $A$ be a $C^*$-algebra satisfying \textup{(H3)}, and let $B\subseteq A$ be a $C^*$-subalgebra.
Suppose there is a CP contraction $E\colon A\to B$ with $E(x)=x$ for $x\in B$. Then
$B$ satisfies \textup{(H3)}.
\end{lemma}

\begin{proof}
Let $T\colon\linf\to B$ be linear, and let $\iota\colon B\to A$ be the inclusion. Since $E$ is a
CP contraction, (D3) gives $\|E\|_{\dec}\le1$. Since $T=E\circ\iota\circ T$,
(D2) and (H3) for $A$ give
\[
\|T\|_{\dec}=\|E\circ(\iota\circ T)\|_{\dec}\le\|\iota\circ T\|_{\dec}=\|\iota\circ T\|_{\cb}=\|T\|_{\cb}.
\]
The reverse inequality is (D1).
\end{proof}

\begin{corollary}\label{cor:heredity}
Let $M$ be a von Neumann algebra satisfying \textup{(H3)}. Then \textup{(H3)} holds for each of the
following algebras:
\begin{enumerate}
\item[(a)] every corner $fMf$, where $f\in M$ is a projection;
\item[(b)] if $M$ has a faithful normal tracial state $\tau$, every von Neumann subalgebra $N\subseteq M$;
\item[(c)] every von Neumann algebra $*$-isomorphic to $M$.
\end{enumerate}
\end{corollary}

Let $N$ be a von Neumann algebra on a Hilbert space $H$, containing a system of $4\times4$ matrix
units $\{e_{ij}\}_{i,j=1}^4$ with $\sum_{i=1}^4e_{ii}=1$.  
Put $Q=e_{11}Ne_{11}$, a von Neumann algebra on $e_{11}H$ with unit $e_{11}$. For a von
Neumann subalgebra $C\subseteq Q$ with $e_{11}\in C$, define
\[
M_4(C):=\left\{\sum_{i,j}e_{i1}c_{ij}e_{1j}: c_{ij}\in C\right\}\subseteq N.
\]

\begin{lemma}\label{lem:matrix}
In the situation above, let $y_1,y_2,y_3,y_4\in Q$ be self-adjoint, some of which may be $0$. 
Let $C=\Wst(y_1,y_2,y_3,y_4)$ be the von Neumann algebra on $e_{11}H$ generated by $y_1,y_2,y_3,y_4$ and
the unit $e_{11}$ of $Q$. Fix four distinct real numbers $\lambda_1,\lambda_2,\lambda_3,\lambda_4$; for instance
$\lambda_i=i$. Put
\[
D=\sum_{i=1}^4\lambda_i\,e_{ii},\qquad
Y=\sum_{i=1}^4e_{i1}y_ie_{1i}+\sum_{i=1}^3\bigl(e_{i,i+1}+e_{i+1,i}\bigr).
\]
Then $D$ and $Y$ are self-adjoint, and
\[
M_4(C)=\Wst(D,Y)=\Wst(D+\iu Y).
\]
In particular $M_4(C)$ is singly generated.
\end{lemma}

\begin{proof}
Let $\mathcal A=\Wst(D,Y)$.
Then $\mathcal A\subseteq M_4(C)$.

We will show the reverse inclusion.
Note that 
\[
e_{ii}=\prod_{j\ne i}\frac{D-\lambda_j}{\lambda_i-\lambda_j}\in\mathcal A\qquad(1\le i\le4).
\]
Fix $1\le i\le3$. We compute $e_{ii}Ye_{i+1,i+1}$ term by term.
\begin{itemize}
\item \emph{Diagonal terms.} For each $k$,
\[
e_{ii}(e_{k1}y_ke_{1k})e_{i+1,i+1}=\delta_{ik}\delta_{k,i+1}\,e_{i1}y_ke_{1,i+1}=0,
\]
since $k$ cannot equal both $i$ and $i+1$.
\item \emph{Off-diagonal terms.} For each $k$,
\[
e_{ii}e_{k,k+1}e_{i+1,i+1}=\delta_{ik}\,e_{i,i+1},\qquad
e_{ii}e_{k+1,k}e_{i+1,i+1}=\delta_{i,k+1}\delta_{k,i+1}\,e_{i,i+1}=0,
\]
since $k+1=i$ and $k=i+1$ cannot hold simultaneously.
\end{itemize}
Hence $e_{ii}Ye_{i+1,i+1}=e_{i,i+1}\in\mathcal A$. 
Since every matrix unit is a product of these and their
adjoints, all $e_{ij}$ belong to $\mathcal A$.
Moreover,
for each $i$ we have $e_{ii}Ye_{ii}=e_{i1}y_ie_{1i}$. Hence
\[
y_i=e_{11}y_ie_{11}=e_{1i}(e_{ii}Ye_{ii})e_{i1}\in\mathcal A\qquad(1\le i\le4).
\]
Hence $\mathcal{A}$ contains $C=\Wst(y_1,y_2,y_3,y_4)$ and all $e_{ij}$,
so $e_{i1}xe_{1j}\in\mathcal{A}$ for all $x\in C$ and all $i,j$.
Therefore $M_4(C)=\mathcal{A}$.
\end{proof}


\section{The finite case}\label{sec:finite}

\begin{proposition}\label{prop:finite}
Let $N$ be a von Neumann algebra of type II$_1$ with a faithful normal tracial state $\tau$. If $N$
satisfies \textup{(H3)}, then $N$ is injective.
\end{proposition}

\begin{proof}
We may regard $N$ as acting on some Hilbert space $H$. 
By \cite[Proposition V.1.35]{Tak},
take a system $\{e_{ij}\}_{i,j=1}^4$
of $4\times4$ matrix units in $N$ with $\sum_{i=1}^4e_{ii}=1$. 
Put $Q=e_{11}Ne_{11}$.

By Lemma~\ref{lem:locality} it suffices to prove that $\Wst(u_1,\dots,u_m)$ is injective for arbitrary
$u_1,\dots,u_m\in\U(N)$. 
For $1\le n\le m$ and $1\le i,j\le4$, the element
$e_{1i}u_ne_{j1}$ lies in $Q$. Let $s_1,\dots,s_L$ ($L=32m$) be an enumeration of the self-adjoint
elements
\[
\mathrm{Re}(e_{1i}u_ne_{j1})\quad\text{and}\quad\mathrm{Im}(e_{1i}u_ne_{j1}),
\]
all of which lie in $Q$. For $0\le l\le L$, define
\[
C_l=\Wst(s_1,\dots,s_l)\subseteq Q,\qquad B_l=M_4(C_l)\subseteq N .
\]
Here, $\Wst$ is taken in $Q$, i.e.\ it includes the unit $e_{11}$ of $Q$. In particular
$C_0=\C e_{11}$.
Note that $B_l$ is a von Neumann subalgebra of $N$ containing the unit of $N$, and
$e_{11}B_le_{11}=C_l$. 
Since $\tau(e_{11})=1/4$,
$\tau_l:=4\,\tau|_{C_l}$ is
a faithful normal tracial state on $C_l$.
Since $C_l$ is finitely generated and $\tau_l$ is faithful, $L^2(C_l,\tau_l)$ is separable, so $C_l$ has separable predual.
Let
$\pi_l$ denote the GNS representation of $C_l$ on $L^2(C_l,\tau_l)$. 

Now we claim that 
$B_l$ (equivalently, $C_l$) is injective for every $0\le l\le L$.
We argue by induction on $l$.

\emph{Base case.} Since $C_0=\C e_{11}$, $B_0=\mathrm{span}\{e_{ij}\}\cong M_4(\C)$. A finite-dimensional
von Neumann algebra is injective.

\emph{Inductive step.} Let $1\le l\le L$ and suppose that $B_{l-1}$ is injective.

\smallskip
(a) \emph{$C_{l-1}$ is injective.} Let $P\colon B(H)\to B_{l-1}$ be a projection of norm one. Identify
$B(e_{11}H)$ with $e_{11}B(H)e_{11}$. Then $x\mapsto e_{11}P(x)e_{11}$ maps $B(e_{11}H)$ into
$e_{11}B_{l-1}e_{11}=C_{l-1}$. It has norm at most one and fixes $C_{l-1}$, because $P$ fixes
$C_{l-1}\subseteq B_{l-1}$. Hence $C_{l-1}$ is injective.

\smallskip
(b) \emph{$C_{l-1}$ is generated by two self-adjoint elements.}
By (a), $\pi_{l-1}(C_{l-1})$ is an injective von Neumann algebra on a separable Hilbert space.
By Proposition~\ref{prop:K1}, $\pi_{l-1}(C_{l-1})$ is singly generated. Since single generation is invariant
under the normal $*$-isomorphism $\pi_{l-1}$, we get $C_{l-1}=\Wst(z)$ for some $z\in C_{l-1}$. Put
$y_1=\mathrm{Re}\,z$ and $y_2=\mathrm{Im}\,z$. Then $C_{l-1}=\Wst(y_1,y_2)$ (generated in $Q$), i.e.\
$C_{l-1}$ is generated by the two self-adjoint elements $y_1,y_2$.

\smallskip
(c) \emph{$B_l$ is singly generated.}
By (b), $C_l=\Wst(y_1,y_2,s_l)$, which is generated in $Q$ by three self-adjoint
elements.
Lemma~\ref{lem:matrix} with $y_3=s_l$ and $y_4=0$ shows that $B_l=M_4(C_l)$ is singly generated.

\smallskip
(d) \emph{$B_l$ satisfies \textup{(H3)}.} 
Since $N$ has the faithful normal tracial state $\tau$ and $B_l\subseteq N$, Corollary 3.2(b) shows that $B_l$ satisfies (H3).

\smallskip
(e) \emph{$B_l$ is injective.} 
The restriction $\tau|_{B_l}$ is a faithful normal tracial state on $B_l$. By (c), (d)
and Corollary~\ref{cor:single}, $B_l$ is injective. This proves the claim.

\smallskip
By the claim, $B_L$ is injective. Let $P_L:B(H)\to B_L$ be a projection of norm one.
Since
$u_n=\sum_{i,j}e_{i1}(e_{1i}u_ne_{j1})e_{1j}$
and
$e_{1i}u_ne_{j1}\in C_L$,
we have $\Wst(u_1,\dots,u_m)\subseteq B_L$.
By \cite[Proposition V.2.36]{Tak}, 
there exists a unique $\tau$-preserving conditional expectation $E\colon B_L\to\Wst(u_1,\dots,u_m)$. 
Then $E\circ P_L$ is a projection of norm one from $B(H)$ onto
$\Wst(u_1,\dots,u_m)\subseteq B_L$. 
Hence $\Wst(u_1,\dots,u_m)$ is injective. 
\end{proof}


\section{Proof of the Main Theorem}\label{sec:main}

\begin{proof}[Proof of the Main Theorem]
Let $N$ satisfy (H3). By \cite[Theorem V.1.19]{Tak}, 
there are central projections $z_\infty$, $z_{\mathrm I}$ and $z_{\mathrm{II}}$
with sum $1$ such that:
\begin{itemize}
\item $z_\infty N$ is properly infinite (the sum of the type II$_\infty$ and III parts and of the
properly infinite type I part);
\item $z_{\mathrm I}N$ is finite of type I;
\item $z_{\mathrm{II}}N$ is of type II$_1$.
\end{itemize}
Each summand satisfies (H3) by Corollary~\ref{cor:heredity}(a).
\begin{enumerate}
\item The summand $z_{\mathrm I}N$ is injective, being of type I.
\item Since $z_\infty N$ is properly infinite,
$M_k(z_\infty N)\cong z_\infty N$ for every $k$.
Hence $M_k(z_\infty N)$ satisfies (H3) for every $k$ by Corollary 3.2(c), and $z_\infty N$ is injective by Proposition~\ref{prop:pi}.
\item By \cite[Corollary V.2.9]{Tak}, $z_{\mathrm{II}}N=\bigoplus_\alpha N_\alpha$,
a direct sum of central summands $N_\alpha$, each of which is $\sigma$-finite and of type II$_1$.
Each $N_\alpha$ has a faithful normal tracial state. 
Since $N_\alpha$ satisfies (H3) by Corollary~\ref{cor:heredity}(a),
$N_\alpha$ is injective by Proposition~\ref{prop:finite}.
\end{enumerate}
Therefore, $N$ is injective.
\end{proof}


\section{Remarks}\label{sec:remarks}

\begin{remark}
Since the converse is the Wittstock--Paulsen theorem
(\cite{Wittstock}, \cite{Paulsen}), the Main Theorem shows that $N$ is injective if and
only if $\|T\|_{\dec}=\|T\|_{\cb}$ for all $T:\ell^\infty_3\to N$. This is optimal, because the identity
holds for every $N$ when $\ell^\infty_3$ is replaced by $\ell^\infty_2$ \cite[Proposition~3.4]{Haa85}.
\end{remark}

\begin{remark}
Proposition~\ref{prop:K1} is applied only to the algebras $C_{l-1}$, which are known to be injective by the
induction hypothesis. The single generator of $B_l$, which is not yet known to be injective, is
constructed explicitly in Lemma~\ref{lem:matrix}.
\end{remark}

\begin{remark}\label{rem:OAI}
By \cite{OAIgen}, every von Neumann algebra with separable predual is singly
generated. Granting this, the Main Theorem has a short proof, which bypasses the induction of
Section~\ref{sec:finite}, Lemma~\ref{lem:matrix} and Proposition~\ref{prop:pi}. 

Let $N$ satisfy (H3).
Choose a net $(p_\alpha)$ of $\sigma$-finite projections in $N$ such that $p_\alpha\nearrow 1$.
Set $N_\alpha=p_\alpha Np_\alpha$
and let $\varphi_\alpha$ be a faithful
normal state on $N_\alpha$.
By Corollary \ref{cor:heredity}(a), $N_\alpha$ satisfies (H3).
For a finite set $F\subset\U(N_\alpha)$, let $N_F$ be the von Neumann algebra generated
by $\sigma^{\varphi_\alpha}_t(F)$, $t\in\mathbb Q$.
By \cite[Theorem~IX.4.2]{Tak2}, $N_F$ is the range of a normal conditional expectation, so it
satisfies (H3), and it has separable predual.
By \cite{OAIgen}, it is singly generated, and by 
Corollary~\ref{cor:single},
$N_F$ is injective.
By \cite[Proposition~5.7]{EL}, $N_\alpha$ is injective,
and so is $N$.

We keep the longer proof because it does not depend on \cite{OAIgen}.
\end{remark}

\subsection*{Acknowledgment}
This work was carried out during the author's stay at the Institut Henri Poincar\'e (IHP) for the programme on Operator Algebras: Approximation, Rigidity and Dynamics. The author is grateful to the IHP for its hospitality.

\subsection*{AI statement}
This work was carried out in collaboration with an AI system (Claude Opus 5.5). The main argument was developed in discussions between the author and the AI system, and the manuscript was drafted mainly by the AI system. The author has verified all proofs and takes full responsibility for the content of this paper.



\begin{thebibliography}{99}

\bibitem{Connes}
A. Connes, \emph{Classification of injective factors. Cases II$_1$, II$_\infty$, III$_\lambda$, $\lambda\neq1$},
Ann.\ of Math.\ (2) \textbf{104} (1976), 73--115.

\bibitem{EL}
E. G. Effros, E. C. Lance, \emph{Tensor products of operator algebras},
Adv.\ Math.\ \textbf{25} (1977), 1--34.

\bibitem{Haa85}
U. Haagerup, \emph{Injectivity and decomposition of completely bounded maps},
in: Operator Algebras and their Connections with Topology and Ergodic Theory (Bu\c{s}teni, 1983),
Lecture Notes in Math.\ \textbf{1132}, Springer, Berlin, 1985, pp.~170--222.

\bibitem{Hayes}
B. Hayes, \emph{On the generator problem for $W^*$-bundles}, arXiv:2609.16428.

\bibitem{JLM}
M. Junge, C. Le Merdy, \emph{Factorization through matrix spaces for finite rank operators between
$C^*$-algebras}, Duke Math.\ J.\ \textbf{100} (1999), no. ~2, 299--319.

\bibitem{OAIgen}
OpenAI, \emph{Relative generation and the generator problem for finite factors}, preprint (announced), September 23, 2026.

\bibitem{Paulsen}
V.I. Paulsen, \emph{Completely bounded maps on $C^*$-algebras and invariant operator ranges}, Proc.\ Amer.\ Math.\ Soc. \textbf{86} (1982), no.~1, 91--96.

\bibitem{Pearcy}
C. Pearcy, \emph{$W^*$-algebras with a single generator}, Proc.\ Amer.\ Math.\ Soc.\ \textbf{13} (1962), 831--832.

\bibitem{Pis19}
G.~Pisier, \emph{On a characterization of the weak expectation property (WEP)}, arXiv:1907.06224.

\bibitem{Pis20}
G. Pisier, 
\emph{Tensor products of C$^*$-algebras and operator spaces---the Connes--Kirchberg problem},
London Math.\ Soc.\ Stud.\ Texts, \textbf{96},
Cambridge Univ.\ Press, Cambridge, 2020.


\bibitem{SuzukiSaito}
N. Suzuki, T. Sait\^o, \emph{On the operators which generate continuous von Neumann algebras},
T\^ohoku Math.\ J.\ (2) \textbf{15} (1963), 277--280.

\bibitem{Tak}
M. Takesaki, \emph{Theory of Operator Algebras I}, reprint of the first (1979) edition, 
Encyclopaedia of Mathematical Sciences, 124; Operator Algebras and Non-commutative Geometry,  5, Springer, Berlin, 2002.

\bibitem{Tak2}
M. Takesaki, \emph{Theory of operator algebras. II}, Encyclopaedia of Mathematical Sciences, 125; Operator Algebras and Non-commutative Geometry, 6, Springer, Berlin, 2003.


\bibitem{Willig}
P. Willig, \emph{Generators and direct integral decompositions of $W^*$-algebras},
T\^ohoku Math.\ J.\ (2) \textbf{26} (1974), 35--37.

\bibitem{Wittstock}
G. Wittstock, \emph{Ein operatorwertiger Hahn-Banach Satz},
J.\ Funct.\ Anal.  \textbf{40} (1981), no.~2, 127--150.


\bibitem{Wogen}
W. R. Wogen, \emph{On generators for von Neumann algebras}, Bull.\ Amer.\ Math.\ Soc.\ \textbf{75} (1969), 95--99.


\end{thebibliography}
\end{document}